\documentclass[12pt,letterpaper]{amsart}
\usepackage{amsmath,txfonts}
\usepackage{amssymb}
\usepackage{amsxtra}
\usepackage{amsthm, color}
\usepackage{txfonts}
\usepackage{graphicx}
\usepackage{times}
\usepackage{citeref}
\usepackage{tikz}
\usepackage{hyperref}
\usepackage{stmaryrd}
\usepackage[T3,T1]{fontenc}
\usepackage{pgfplots}

\usetikzlibrary{calc}

\usepackage{mathrsfs}
\usepackage{amsfonts}
\usepackage{amssymb}
\usepackage{ifthen}
\usepackage{graphicx}
\nonstopmode \numberwithin{equation}{section}

\newtheorem{thm}{Theorem}[section]
\newtheorem{lem}{Lemma}[section]
\newtheorem{cor}[thm]{Corollary}
\newtheorem{prop}[thm]{Proposition}

\newtheorem{step}{Step}[section]

\theoremstyle{definition}
\newtheorem{mlem}{Main lemma}[section]
\newtheorem{assertion}{Assertion}[section]
\newtheorem{cl}{Claim}[section]
\newtheorem{ca}{Case}[section]
\newtheorem{sca}{Subcase}[section]
\newtheorem{scl}{Subclaim}[section]
\newtheorem{conj}[thm]{Conjecture}
\newtheorem{fact}{Fact}[section]
\newtheorem{defn}[thm]{Definition}
\newtheorem{op}[thm]{Open Problem}
\newtheorem{prob}{Problem}[section]
\newtheorem{ques}{Question}[section]
\newtheorem{rem}[thm]{Remark}
\newtheorem{exam}[thm]{Example}

\numberwithin{equation}{section}

\newcounter {own}
\def\theown {\thesection       .\arabic{own}}

\newenvironment{pf}[1][]{%
 \vskip 3mm
 \noindent
 \ifthenelse{\equal{#1}{}}%
  {{\slshape Proof. }}%
  {{\slshape #1.} }%
 }%
{\qed\bigskip}

\newcounter{alphabet}
\renewcommand{\thealphabet}{\Alph{alphabet}}

\newenvironment{Thm}[1][]{\refstepcounter{alphabet}%
	\bigskip
	\noindent
	{\bf Theorem \thealphabet}%
	\ifthenelse{\equal{#1}{}}{}{ (#1)}%
	{\bf .} \itshape
}{\vskip 8pt}

\newenvironment{Lem}[1][]{\refstepcounter{alphabet}%
	\bigskip
	\noindent
	{\bf Lemma \thealphabet}%
	{\bf .} \itshape
}{\vskip 8pt}

\newcommand{\CC}{{\mathcal C}}

\def\be{\begin{equation}}
\def\ee{\end{equation}}

\newcommand{\ben}{\begin{enumerate}}
\newcommand{\een}{\end{enumerate}}

\newcommand{\blem}{\begin{lem}}
\newcommand{\elem}{\end{lem}}
\newcommand{\bthm}{\begin{thm}}
\newcommand{\ethm}{\end{thm}}
\newcommand{\bcor}{\begin{cor}}
\newcommand{\ecor}{\end{cor}}
\newcommand{\beg}{\begin{exam}}
\newcommand{\eeg}{\end{exam}}
\newcommand{\begs}{\begin{examples}}
\newcommand{\eegs}{\end{examples}}
\newcommand{\bdefe}{\begin{defn}}
\newcommand{\edefe}{\end{defn}}
\newcommand{\bprob}{\begin{prob}}
\newcommand{\eprob}{\end{prob}}
\newcommand{\bques}{\begin{ques}}
\newcommand{\eques}{\end{ques}}
\newcommand{\bei}{\begin{itemize}}
\newcommand{\eei}{\end{itemize}}
\newcommand{\bcon}{\begin{conj}}
\newcommand{\econ}{\end{conj}}
\newcommand{\bop}{\begin{op}}
\newcommand{\eop}{\end{op}}

\newcommand{\bas}{\begin{assertion}}
\newcommand{\eas}{\end{assertion}}

\newcommand{\bfa}{\begin{fact}}
\newcommand{\efa}{\end{fact}}

\newcommand{\bca}{\begin{ca}}
\newcommand{\eca}{\end{ca}}

\newcommand{\bst}{\begin{step}}
\newcommand{\est}{\end{step}}

\newcommand{\bsca}{\begin{sca}}
\newcommand{\esca}{\end{sca}}

\newcommand{\bcl}{\begin{cl}}
\newcommand{\ecl}{\end{cl}}

\newcommand{\bmlem}{\begin{mlem}}
\newcommand{\emlem}{\end{mlem}}

\newcommand{\bscl}{\begin{scl}}
\newcommand{\escl}{\end{scl}}

\newcommand{\bcons}{\begin{conjs}}
\newcommand{\econs}{\end{conjs}}

\newcommand{\bprop}{\begin{prop}}
\newcommand{\eprop}{\end{prop}}

\newcommand{\br}{\begin{rem}}
\newcommand{\er}{\end{rem}}
\newcommand{\brs}{\begin{rems}}
\newcommand{\ers}{\end{rems}}
\newcommand{\bo}{\begin{obser}}
\newcommand{\eo}{\end{obser}}
\newcommand{\bos}{\begin{obsers}}
\newcommand{\eos}{\end{obsers}}
\newcommand{\bpf}{\begin{pf}}
\newcommand{\epf}{\end{pf}}
\newcommand{\ba}{\begin{array}}
\newcommand{\ea}{\end{array}}
\newcommand{\beq}{\begin{eqnarray}}
\newcommand{\beqq}{\begin{eqnarray*}}
\newcommand{\eeq}{\end{eqnarray}}
\newcommand{\eeqq}{\end{eqnarray*}}

\newcounter{minutes}
\divide\time by 60
\newcounter{hours}
\multiply\time by 60 \addtocounter{minutes}{-\time}

\begin{document}

\def\thefootnote{}
\footnotetext{ \texttt{\tiny File:~\jobname .tex,
           printed: \number\year-\number\month-\number\day,
           \thehours.\ifnum\theminutes<10{0}\fi\theminutes}
} \makeatletter\def\thefootnote{\@arabic\c@footnote}\makeatother

\bibliographystyle{amsplain}
\title []
{The sharp Lipschitz continuity problem with respect to the pseudo hyperbolic metric}

\author{Shaolin Chen}
\address{S. L. Chen,    Center for Applied Mathematics of Guangxi, Guangxi Normal University,
Guilin, Guangxi 541004, People's Republic of China} \email{mathechen@126.com}

\author{Hidetaka Hamada}
\address{H. Hamada, Industry-Academia Co-innovation and Research Promotion Headquarters,
Kyushu Sangyo University,
3-1 Matsukadai 2-Chome, Higashi-ku, Fukuoka 813-8503, Japan.}
\email{hi.hamada01@gmail.com}

\author{Manzi Huang}
\address{M. Z.  Huang, Key Laboratory of High Performance Computing and Stochastic Information Processing,
College of Mathematics and Statistics, Hunan Normal University, Changsha, Hunan 410081, People's Republic of China}
\email{mzhuang@hunnu.edu.cn}

\author{Lei Jin}
\address{L. Jin, Key Laboratory of Computing and Stochastic Mathematics (Ministry of Education), 
School of Mathematics and Statistics, Hunan Normal University, Changsha, Hunan 410081, P. R. China}
\email{leijin0825@163.com}

\subjclass[2020]{Primary 30H30, 30C62, 31A05.}
\keywords{Harmonic Bloch type mapping, $(K,K_{0})$-Quasiregular Mapping, Lipschitz continuous, pseudo hyperbolic metric}

\begin{abstract}
 The main purpose of this paper is to address an open problem posed by Huang, Rasila, and Zhu in [Anal. Math. 48(2022), 1069-1080]. We first establish a sharp Lipschitz continuity result for locally univalent harmonic Bloch mappings with respect to the pseudo hyperbolic metric. This result is then applied to demonstrate that harmonic $(K,K_0)$-quasiregular Bloch type mappings are Lipschitz continuous under the same metric, where $K\geq1$ and $K_{0}\geq0$. Moreover, the estimates obtained are asymptotically sharp as $K \to 1^+$ and $K_0 \to 0^+$. In particular, when $K_0 = 0$, our theorem provides a solution to the aforementioned open problem in the asymptotically sharp sense.
\end{abstract}

\maketitle \pagestyle{myheadings} \markboth{S. L. Chen, H. Hamada, M. Z. Huang and L. Jin}{The sharp Lipschitz continuity problem with respect to the pseudo hyperbolic metric}

\dedicatory{}


\section{Introduction and main results}\label{sec-1}
Let $\mathbb{D} = \{ z : |z| < 1 \}$ denote the open unit disk in the complex plane $\mathbb{C}$, and let $\mathbb{T} := \partial \mathbb{D}$ be its boundary, the unit circle.  
For a subdomain $\Omega$ of $\mathbb{C}$, we denote by $\CC^m(\Omega)$ the class of all complex-valued functions on $\Omega$ that possess continuous derivatives up to order $m$, 
where $m \in  \{0,1,2,\ldots\}$. In particular, we set $\CC(\Omega) = \CC^0(\Omega)$.
For $z=x+iy\in\mathbb{C}$, the complex formal derivatives operators are defined by
$$\frac{\partial}{\partial z}=\frac{1}{2}\left(\frac{\partial }{\partial x}-i\frac{\partial }{\partial y}\right)
~~~\mbox{and}~~~\frac{\partial}{\partial \overline{z}}=\frac{1}{2}\left(\frac{\partial f}{\partial x}+i\frac{\partial f}{\partial y}\right).$$

For $z = r e^{i\vartheta} \in \mathbb{C}$ and $\vartheta \in [0, 2\pi]$, the directional derivative of $f$ at $z$ in the direction $\vartheta$ is defined as
$$
\partial_\vartheta f(z)
=
\lim_{r \to 0^+} \frac{f(z + r e^{i\vartheta}) - f(z)}{r}
=
e^{i\vartheta} f_z(z) + e^{-i\vartheta} f_{\bar z}(z),
$$ where $f_z=\partial f/\partial z$ and $f_{\overline{z}}=\partial f/\partial \overline{z}$.
Consequently, the maximum and minimum moduli of the directional derivative are given respectively by
$$
\Lambda_f(z) := \max_{0 \leq \vartheta \leq 2\pi} |\partial_\vartheta f(z)|
= |f_z(z)| + |f_{\bar z}(z)|,
$$
and
$$
\lambda_f(z) := \min_{0 \leq \vartheta\leq 2\pi} |\partial_\vartheta f(z)|
= \bigl| |f_z(z)| - |f_{\bar z}(z)| \bigr|.
$$


A complex-valued function $f = u + iv \in \mathcal{C}^2(\Omega) $ is called a \emph{harmonic mapping} in the domain \( \Omega \subset \mathbb{C} \) if both its real and imaginary parts are harmonic, i.e.,
$$
\Delta u = \Delta v = 0 \quad \text{in } \Omega,
$$ where $\Delta$ represents the Laplacian operator
 $$\Delta:=4\frac{\partial^{2}}{\partial z\partial \overline{z}}=
 \frac{\partial^{2}}{\partial x^{2}}+\frac{\partial^{2}}{\partial y^{2}}$$
 and $z=x+iy\in\Omega$ (see \cite{Duren2004}).
 
It is a classical result that a complex-valued harmonic function (or  harmonic mapping) \( f \) is locally univalent and sense-preserving in \( \Omega \) if and only if its Jacobian determinant
$$
J_f:= |f_z|^2 - |f_{\bar z}|^2
$$
is strictly positive throughout \( \Omega \). Equivalently, the  second complex dilatation dilatation $\omega_f$ of $ f $ satisfies
$$
|\omega_{f}(z)|
=
\left| \frac{f_{\bar z}(z)}{f_z(z)} \right|
< 1,
\qquad z \in \Omega.
$$
Moreover, if $\Omega$ is a simply connected domain,
then $f$ has the canonical decomposition
$f=h+\overline{g}$, where $h$ and $g$ are analytic in $\Omega$. 

Denote by $\mathcal{A}(\Omega)$ the set of all analytic functions from $\Omega$ into $\mathbb{C}$, 
and by $\mathcal{H}(\Omega)$ the set of all complex-valued harmonic functions from $\Omega$ into $\mathbb{C}$.

A mapping $f:\Omega\to\mathbb{C}$ is said to be \emph{absolutely continuous on lines} (abbreviated as $ACL$) in the domain $\Omega$ if, for every closed rectangle $R\subset\Omega$ whose sides are parallel to the coordinate axes, $f$ is absolutely continuous on almost every horizontal line and almost every vertical line in $R$. Such a mapping possesses partial derivatives $f_x$ and $f_y$ almost everywhere in $\Omega$. Moreover, we say that $f\in ACL^2$ if $f\in ACL$ and its partial derivatives are locally $L^2$-integrable in $\Omega$.

A mapping $f:\Omega\to\mathbb{C}$ is called a \emph{$(K,K_0)$-quasiregular mapping} if it satisfies the following conditions:
\begin{enumerate}
    \item[(1)] $f\in ACL^2$ in $\Omega$ and $J_f>0$ almost everywhere in $\Omega$, where $J_f$ denotes the Jacobian determinant of $f$;
    \item[(2)] there exist constants $K\ge 1$ and $K_0\ge 0$ such that
    \beqq
        \Lambda_f^2 \le K J_f + K_0 \quad \text{a.e. in } \Omega.
    \eeqq
\end{enumerate}

In the special case where $K_0=0$, a $(K,K_0)$-quasiregular mapping is said to be \emph{$K$-quasiregular}. Furthermore, we say that $f$ is \emph{$K$-quasiconformal} in $\Omega$ if $f$ is $K$-quasiregular and homeomorphic in $\Omega$. 
In particular, if a $(K, K_0)$-quasiregular mapping (respectively, $K$-quasiregular mapping) $f$ is harmonic, then it is called a harmonic $(K, K_0)$-quasiregular mapping (respectively, harmonic $K$-quasiregular mapping) (see \cite{FinnSerrin1958,KM,K-15,KL, Nirenberg1953}).

By Lewy's Theorem (\cite[p.20]{Duren2004}), we immediately obtain the following result.

\begin{prop}
Let $f$ be a harmonic $(K,K_0)$-quasiregular mapping in a domain $\Omega \subset \mathbb{C}$, where
$K\ge 1$ and $K_0\ge 0$ are constants. Then $f$ is sense-preserving in $\Omega$ if and only if $f$ is locally univalent in $\Omega$.
\end{prop}




Fix $w \in \mathbb{D}$. Let $\varphi_w$ be the M\"obius transformation of $\mathbb{D}$ defined by
$$
\varphi_w(z) = \frac{w - z}{1 - \overline{w} z}, \qquad z \in \mathbb{D}.
$$
The \emph{automorphism group} of the unit disk $\mathbb{D}$, denoted by $\operatorname{Aut}(\mathbb{D})$, 
is the set of all univalent analytic functions that map  $\mathbb{D}$ onto itself, 
endowed with the operation of function composition.
The pseudo hyperbolic distance on $\mathbb{D}$ is defined by
$$
\rho(z, w) = |\varphi_w(z)|.
$$
This distance is $\operatorname{Aut}(\mathbb{D})$ invariant. More precisely, for any $\phi \in \operatorname{Aut}(\mathbb{D})$,
$$
\rho(\phi(z), \phi(w)) = \rho(z, w).
$$


A complex-valued function $f \in \mathcal{C}^2(\mathbb{D})$ is said to be of \textit{Bloch type} if
$$
\|f\|_{\mathscr{B}_{s}} := \sup_{z \in \mathbb{D}} \mathcal{B}_{f}(z) < \infty,
$$ where $\mathcal{B}_{f}(z)=(1 - |z|^2) \Lambda_f(z).$
The above quantity defines a seminorm on the space of such functions. Endowed with the norm
\[
\|f\|_{\mathscr{B}} := |f(0)| + \sup_{z \in \mathbb{D}} \mathcal{B}_{f}(z),
\]
the resulting normed space, denoted by $\mathscr{B}_{t}$, is called the \textit{Bloch type space}. In particular, we denote by
\[
\mathscr{B}:=\mathscr{B}_{t}\cap\mathcal{A}(\mathbb{D})
\quad\text{and}\quad
\mathscr{B}_{h}:=\mathscr{B}_{t}\cap\mathcal{H}(\mathbb{D})
\]
the analytic Bloch space and the harmonic Bloch space, respectively.
Moreover, we refer to $f\in\mathscr{B}$ and $f\in\mathscr{B}_{h}$ an analytic Bloch function  and a harmonic Bloch mapping, respectively.
For the analytic Bloch space, we refer the reader to \cite{AndersonCluniePommerenke1974} for further details.
Later, Colonna \cite{Colonna1989} showed that if $f \in \mathcal{H}(\mathbb{D})$ satisfies
\be\label{eq-cj-1}
\sup_{\substack{z,w \in \mathbb{D}, ~ z \neq w}} \frac{|f(z) - f(w)|}{\sigma(z,w)} < \infty,
\ee
then $f \in \mathscr{B}_{h}$, where
$$
\sigma(z,w) = \frac{1}{2} \log \left( \frac{1 + \rho(z,w)}{1 - \rho(z,w)} \right) = \operatorname{arctanh}(\rho(z,w))
$$
is the hyperbolic distance between $z$ and $w$ in $\mathbb{D}$.
The converse direction follows immediately, hence  
$f \in \mathcal{H}(\mathbb{D})$ satisfies (\ref{eq-cj-1}) if and only if  
$f \in \mathscr{B}_{h}$ (see \cite{ChenHamadaZhu2022,Colonna1989}).

For $f\in\mathscr{B}$, Ghatage, Yan and Zheng \cite{GhatageYanZheng2000} proved that
$\mathcal{B}_{f}$ is Lipschitz continuous with respect to the pseudo hyperbolic
metric, which is given as follows.

\begin{Thm}{\rm (\cite[Theorem  1]{GhatageYanZheng2000})}\label{Thm-A}
Let $f\in\mathscr{B}$. Then, for all $z_{1}, z_{2}\in\mathbb{D}$,
$$\left|\mathcal{B}_{f}(z_{1})-\mathcal{B}_{f}(z_{2})\right|\leq3.31\|f\|_{\mathscr{B}_{s}}~\rho(z_{1},z_{2}).$$
\end{Thm}

Later,  Hosokawa and Ohno \cite{HO} employed a different approach and proved that, for all $z_{1}, z_{2}\in\mathbb{D}$,
\be\label{eq-cj-2}
\left|\mathcal{B}_{f}(z_{1})-\mathcal{B}_{f}(z_{2})\right|\leq 20\,\|f\|_{\mathscr{B}_{s}}\,\rho(z_{1},z_{2}),
\ee
where $f\in\mathscr{B}$. They used \eqref{eq-cj-2} to study composition operators on Bloch spaces (see \cite{HO, HO-1}).
In \cite{Xiong2003}, Xiong proved the 
 sharp form of Theorem \ref{Thm-A} (or \eqref{eq-cj-2}) as follows.

\be\label{eq-cj-3}
\left|\mathcal{B}_{f}(z_{1})-\mathcal{B}_{f}(z_{2})\right|\leq \frac{3\sqrt{3}}{2}\,\|f\|_{\mathscr{B}_{s}}\,\rho(z_{1},z_{2}).
\ee

Recently, an analog of Theorem \ref{Thm-A} was established by Huang, Rasila, and Zhu \cite{HuangRasilaZhu2022}, 
who investigated the Lipschitz continuity of harmonic Bloch mapping with respect to the pseudo hyperbolic metric.
It  reads as follows. 
For $f\in\mathscr{B}_{h}$, we have
\be\label{eq-cj-4}\left|\mathcal{B}_{f}(z_{1})-\mathcal{B}_{f}(z_{2})\right|\leq3\sqrt{3}\|f\|_{\mathscr{B}_{s}}~\rho(z_{1},z_{2})\ee
for all $z_{1}, z_{2}\in\mathbb{D}$.

In order to generalize the analytic Bloch spaces, Efraimidis, Gaona, Hern\'andez and Venegas \cite{EfraimidisGaonaHernandezVenegas2017} 
introduced the class of harmonic Bloch type mappings as follows. A complex-valued harmonic function $ f $ defined in $\mathbb{D}$ is said to be a harmonic Bloch type mapping if  
$$
\|f\|_{\mathscr{B}_{h^*,s}} = \sup_{z \in \mathbb{D}} \mathcal{B}_{f}^{\ast}(z)< \infty,
$$  where $\mathcal{B}_{f}^{\ast}(z)=(1 - |z|^2) \sqrt{|J_f(z)|}$.
We denote this class by \(\mathscr{B}_h^*\). The quantity  
$$
\|f\|_{\mathscr{B}_h^*} = |f(0)| + \|f\|_{\mathscr{B}_{h^*,s}}
$$ 
is referred to as the Bloch type pseudo-norm of $ f $. It is not immediately obvious that \(\mathscr{B}_h \subseteq \mathscr{B}_h^*\).
Huang, Rasila, and Zhu \cite{HuangRasilaZhu2022} established that harmonic $K$-quasiregular Bloch 
type mappings are Lipschitz continuous with respect to the pseudo hyperbolic metric. Their result is stated as follows.

\begin{Thm}\rm (\cite[Theorem  1.5]{HuangRasilaZhu2022})\label{Thm-Cl}
Let $f\in\mathscr{B}_h^*$ be a locally univalent harmonic $K$-quasiregular mapping in $\mathbb D$. Then, for all $z_{1},z_{2}\in\mathbb D$,
\be\label{eq-CJ-5}
\left|
\mathcal{B}_{f}^{\ast}(z_{1})
-
\mathcal{B}_{f}^{\ast}(z_{2})
\right|
\le
2.8587(K+1)\rho(z_{1},z_{2})\|f\|_{ \mathscr{B}_{h^*,s}}.
\ee

\end{Thm}

Since inequalities (\ref{eq-cj-4}) and (\ref{eq-CJ-5}) are not sharp, Huang, Rasila, and Zhu posed the following open problem.

\begin{op}{\rm (\cite[Question 1.6]{HuangRasilaZhu2022})}\label{Open-1}
What are the optimal constants in inequalities (\ref{eq-cj-4}) and (\ref{eq-CJ-5})?
\end{op}

The sharp constant of inequality \eqref{eq-cj-4} was obtained by Chen, Hamada and Zhu (see \cite[Theorem 2.2]{ChenHamadaZhu2022})
as follows.
\begin{align}\label{eq1.6}
\left|\mathcal{B}_{f}(z_{1})-\mathcal{B}_{f}(z_{2})\right|\leq\frac{3\sqrt{3}}{2}\|f\|_{\mathscr{B}_{s}}~\rho(z_{1},z_{2})
\end{align}
for all $z_{1}, z_{2}\in\mathbb{D}$.
However, the sharp constant of inequality \eqref{eq-CJ-5} remains open.

Also, in view of \eqref{eq1.6} and Theorem B,
the following question would be interesting.

\begin{ques}\label{ques-1}
Can we find the sharp inequality similar to \eqref{eq1.6}
in the case $f\in\mathscr{B}_h$ is locally univalent?
\end{ques}

The primary aim of this paper is to investigate Open Problem \ref{Open-1} and Question \ref{ques-1}. 
We first establish a sharp Lipschitz continuity result for locally univalent harmonic Bloch mappings with respect to the pseudo hyperbolic metric,
which gives a positive answer to Question \ref{ques-1}. 
The result is stated as follows.

\begin{thm}\label{CL-lem-2}
Let $f\in\mathscr{B}_h$ be locally univalent. Then, for all $z_{1},~z_{2}\in\mathbb{D}$,
\be\label{eq-cl-10}\left|\mathcal{B}_{f}(z_{1})-\mathcal{B}_{f}(z_{2})\right|\leq\|f\|_{\mathscr{B}_{s}}C_{0}~\rho(z_{1},z_{2}),\ee
where $C_{0}=(4+2\sqrt{5})e^{-\frac{(1+\sqrt{5})}{2}}.$
  The constant $C_0$ in (\ref{eq-cl-10}) is sharp.
\end{thm}


Applying Theorem~\ref{CL-lem-2}, we relax the assumption of harmonic \(K\)-quasiregularity employed in Theorem~\ref{Thm-Cl} by replacing it with the broader condition of harmonic \((K,K_{0})\)-quasiregularity, where \(K \ge 1\) and \(K_{0} \ge 0\) are fixed constants. Consequently, we obtain the following asymptotically sharp result.

\begin{thm}\label{thm-1}
Let $f\in\mathscr{B}_h^*$ be a locally univalent harmonic $(K,K_{0})$-quasiregular mapping in $\mathbb D$,
where $K \ge 1$ and $K_{0} \ge 0$ are constants. Then, for all $z,~w\in\mathbb D$,
\be\label{eq-m21}
\left|
\mathcal{B}_{f}^{\ast}(z)-\mathcal{B}_{f}^{\ast}(w)\right|
\leq C(K,K_{0})\rho(z,w),
\ee
where $$C(K,K_{0})=
\left[C_{0}\frac{\sqrt {K+3}}{2}
+
\left(
\frac{3\sqrt{3}}{2}
+
\operatorname{coth}2
\right)
 \frac{\sqrt{K - 1}}{2}\right]\|f\|_{\mathscr{B}_{h^*,s}}+\left[
C_{0}+\left(\frac{3\sqrt{3}}{2}+\operatorname{coth}2\right)\right]
\frac{\sqrt{K_{0}}}{2}
$$
and  $C_{0}$ is the same as in Theorem \ref{CL-lem-2}.
 Moreover, inequality (\ref{eq-m21}) is asymptotically sharp
 as $(K,K_0) \to (1,0)$, and  
$$
\lim_{(K,K_0) \to (1,0)} C(K,K_0) = C_0.
$$
\end{thm}

In particular, by setting $K_{0} = 0$ in Theorem~\ref{thm-1}, we obtain the following result
by using the method of the proof for Theorem~\ref{thm-1}.
This result resolves the open problem stated in the open problem \ref{Open-1} for the equation \eqref{eq-CJ-5}, in the asymptotically sharp sense.

\begin{thm}\label{cor-1}
Let $f\in\mathscr{B}_h^*$ be a locally univalent harmonic $K$-quasiregular mapping in $\mathbb D$,
where $K \ge 1$ is a constant. Then, for all $z,~w\in\mathbb D$,
\be\label{eq-m22}
\left|
\mathcal{B}_{f}^{\ast}(z)-\mathcal{B}_{f}^{\ast}(w)\right|
\leq C(K)\|f\|_{\mathscr{B}_{h^*,s}}\rho(z,w),
\ee where $k=(K-1)/(K+1)$,
$$C(K)=
C_{0}\frac{(K+1)}{2\sqrt K}
+
\frac{kL(K)}{2}\left(\frac{3\sqrt{3}}{2}
+
\operatorname{coth}L(K)\right)
\frac{(K-1)}{2\sqrt{K}},
$$

$$
L(K)=\frac{2}{1+\sqrt{1-k^2}},
$$
and  $C_{0}$ is the same as in Theorem \ref{CL-lem-2}.
 Moreover, inequality (\ref{eq-m22}) is asymptotically sharp
as $K \to 1^{+}$, and  
$$
\lim_{K \to 1^{+}} C(K) = C_0.
$$
\end{thm}

\begin{rem}
Elementary calculations show that 
\beqq
C(K)&\leq& C_{0}\frac{(1+K)}{2}+\left(
\frac{3\sqrt{3}}{2}
+
\operatorname{coth}2
\right)\frac{(1+K)}{4\sqrt{2}}\\
&=&(1+K)\left(\frac{C_{0}}{2}+\frac{3\sqrt{3}}{8\sqrt{2}}
+\frac{\operatorname{coth}2}{4\sqrt{2}}\right)\\
&\approx&1.4824(1+K),
\eeqq which yields that 
$$C(K)<2.8587(K+1).$$
Consequently, Theorem \ref{cor-1} is also an improvement  of Theorem \ref{Thm-Cl}.
\end{rem}


The proofs of Theorems \ref{CL-lem-2}, \ref{thm-1} and  \ref{cor-1}
will be given in Section \ref{sec-2}.


\section{The proofs of the main results}\label{sec-2}


\subsection*{Proof of Theorem \ref{CL-lem-2}} Prior to proving this theorem, we shall recall or prove some necessary Lemmas.

\begin{Lem}{\rm (\cite[Corollary 2]{BMY-96})}\label{Im-1}
Let $f\in\mathscr{B}$ be a locally univalent analytic function in $\mathbb{D}$ satisfying $\|f\|_{ \mathscr{B}_{s}}\leq1$, $f(0)=0$ and $f'(0)=\alpha\in(0,1]$.
Then, for $z\in\mathbb{D}$,
$$\mathcal{B}_{f}(z)\geq(1+m(\alpha))\frac{1+|z|}{1-|z|}e^{\left[1-(1+m(\alpha))\frac{1+|z|}{1-|z|}\right]}$$
with equality at $z=re^{i\theta}$, $r\in(0,1)$, if and only if $f(z)=e^{i\theta}F_{\alpha}^{0}(e^{-i\theta}z)$,
where $$(1+m(\alpha))e^{-m(\alpha)}=\alpha$$ and 
$$F_{\alpha}^{0}(z)=-\frac{1}{2}e^{\left[1-(1+m(\alpha))\frac{1+z}{1-z}\right]}+\frac{1}{2}e^{-m(\alpha)}.$$
\end{Lem}

\begin{lem}\label{CL-lem-1}
Let $f\in\mathscr{B}$ be a locally univalent analytic function in $\mathbb{D}$ satisfying $\|f\|_{ \mathscr{B}_{s}}\leq1$ and $f'(0)=\alpha\in(0,1]$.
Then, for $z\in\mathbb{D}$,
$$\mathcal{B}_{f}(z)\geq(1+m(\alpha))\frac{1+|z|}{1-|z|}e^{\left[1-(1+m(\alpha))\frac{1+|z|}{1-|z|}\right]}$$
with equality at $z=re^{i\theta}$, $r\in(0,1)$, if and only if $f(z)=e^{i\theta}F_{\alpha}(e^{-i\theta}z)+C$,
where \be\label{eq-CJ-12}(1+m(\alpha))e^{-m(\alpha)}=\alpha,\ee 
$$F_{\alpha}(z)=-\frac{1}{2}e^{\left[1-(1+m(\alpha))\frac{1+z}{1-z}\right]}$$
and $C$ is a constant.
\end{lem}
\bpf For $f\in\mathscr{B}$, let $$f_{1}(z)=f(z)-f(0),~~z\in\mathbb{D}.$$
Then  $\|f_{1}\|_{ \mathscr{B}_{s}}=\|f\|_{ \mathscr{B}_{s}}\leq1$, $f_{1}(0)=0$ and $f_{1}'(0)=\alpha\in(0,1]$.
Applying Lemma \ref{Im-1} to $f_{1}$ yields the desired result. 
\epf

\blem\label{lem2}
For all \(a,x\in\mathbb{R}\), we have
\begin{align*}
&(1+x)^2(x+2)
\Bigl[
-2(1+x)^2a^3
+
(3x^2+9x+4)a^2
+
(x^2+2x+4)a
-
(x+2)
\Bigr]
\\
&\quad
-
x[a(1+x)-1]
\Bigl[
a(2x^3+8x^2+13x+8)-x-2
\Bigr]
\\
&=
-N(a,x)
\Bigl[
2a(1+x)^3+x^3+4x^2+6x+2
\Bigr],
\end{align*}
where
$
N(a,x)
=
(x^2+3x+2)a^2
-
(x^2+6x+6)a
+
x+2.
$
\elem

\begin{proof}
Expanding the left-hand side and collecting the powers of \(a\), we
obtain
\begin{align*}
(-2x^5-12x^4-28x^3-32x^2-18x-4)a^3
&+
(x^5+11x^4+34x^3+46x^2+30x+8)a^2
\\
&\quad
+
(x^5+8x^4+26x^3+44x^2+34x+8)a
\\
&\quad
-x^4-6x^3-14x^2-14x-4.
\end{align*}

On the other hand,
\begin{align*}
-N(a,x)
\Bigl[
2a(1+x)^3+x^3+4x^2+6x+2
\Bigr]
&=
-\Bigl[
(x^2+3x+2)a^2
-
(x^2+6x+6)a
+
x+2
\Bigr]
\\
&\qquad\times
\Bigl[
2a(1+x)^3+x^3+4x^2+6x+2
\Bigr].
\end{align*}
Collecting the powers of \(a\), this becomes
\begin{align*}
&
-2(x^2+3x+2)(1+x)^3a^3
\\
&\quad
+
\Bigl[
2(x^2+6x+6)(1+x)^3
-
(x^2+3x+2)(x^3+4x^2+6x+2)
\Bigr]a^2
\\
&\quad
+
\Bigl[
(x^2+6x+6)(x^3+4x^2+6x+2)
-
2(x+2)(1+x)^3
\Bigr]a
\\
&\quad
-
(x+2)(x^3+4x^2+6x+2).
\end{align*}
The coefficient of \(a^3\) is
\beqq
-2(x^2+3x+2)(1+x)^3
=
-2x^5-12x^4-28x^3
-32x^2-18x-4.
\eeqq
The coefficient of \(a^2\) is
\begin{align*}
&2(x^2+6x+6)(1+x)^3
-
(x^2+3x+2)(x^3+4x^2+6x+2)
\\
&=
x^5+11x^4+34x^3+46x^2+30x+8.
\end{align*}
The coefficient of \(a\) is
\begin{align*}
&(x^2+6x+6)(x^3+4x^2+6x+2)
-
2(x+2)(1+x)^3
\\
&=
x^5+8x^4+26x^3+44x^2+34x+8.
\end{align*}
Finally, the last term is
\beqq
-(x+2)(x^3+4x^2+6x+2)
=
-x^4-6x^3-14x^2-14x-4.
\eeqq
Therefore, the right-hand side is equal to
\begin{align*}
(-2x^5-12x^4-28x^3-32x^2-18x-4)a^3
&
+
(x^5+11x^4+34x^3+46x^2+30x+8)a^2
\\
&\quad
+
(x^5+8x^4+26x^3+44x^2+34x+8)a
\\
&\quad
-x^4-6x^3-14x^2-14x-4,
\end{align*}
which is exactly the expansion of the left-hand side. Hence the
identity follows.
\end{proof}

\blem\label{lem1}
For $a\geq1$ and $x>0$, let
\beqq
P(a,x)
=
ae^{1-a}
\left[1-(1+x)e^{-ax}\right]\frac{(x+2)}{x}.
\eeqq
Then
\beqq
\sup_{a\geq1,\ x>0}P(a,x)
=
(4+2\sqrt5)e^{-\left(\frac{1+\sqrt5}{2}\right)}.
\eeqq
\elem

\begin{proof}
Since, for  $a\geq1$ and $x>0$, 
\beqq
\log(1+x)<x\leq ax,
\eeqq
we have
\beqq
(1+x)e^{-ax}<1.
\eeqq
Consequently, $$P(a,x)>0.$$

On the open set
\beqq
D^\circ=\{(a,x):a>1,\ x>0\},
\eeqq
define
\beqq
\ell(a,x)=\log P(a,x).
\eeqq
Then
\beqq
\ell(a,x)
=
\log a+1-a+\log(x+2)-\log x
+
\log\left[1-(1+x)e^{-ax}\right].
\eeqq
Differentiating directly, we obtain
\beqq
\ell_a
=
\frac1a-1+
\frac{x(1+x)e^{-ax}}
{1-(1+x)e^{-ax}}
\eeqq
and
\beqq
\ell_x
=
-\frac{2}{x(x+2)}
+
\frac{[a(1+x)-1]e^{-ax}}
{1-(1+x)e^{-ax}}.
\eeqq

We begin by showing that $P$ has no critical points in
$D^\circ$. Since \(P(a,x)>0\) on \(D^\circ\), we have
\beq\label{eq-0}
P_a(a,x)=P(a,x)\ell_a(a,x)
\eeq
and 
\beqq
P_x(a,x)=P(a,x)\ell_x(a,x).
\eeqq
Thus, a point of \(D^\circ\) is a critical point of \(P\)
if and only if it is a critical point of \(\ell\).
Assume, to the contrary, that
\beqq
\ell_a(a,x)=0,
\qquad
\ell_x(a,x)=0.
\eeqq
The equation $\ell_a=0$ gives
\beq\label{eq-1}
e^{-ax}
=
\frac{a-1}
{(1+x)[a(1+x)-1]}.
\eeq
That is, this relation can be written as $G(a,x)=0$, where
\beqq
G(a,x)
=
ax-
\log
\frac{(1+x)[a(1+x)-1]}{a-1}.
\eeqq
Substituting the expression for $e^{-ax}$ from \eqref{eq-1} into
$\ell_x=0$ yields
\beqq
-\frac{2}{x(x+2)}
+
\frac{(a-1)[a(1+x)-1]}
{ax(1+x)}
=
0.
\eeqq
Multiplying by $ax(1+x)(x+2)$ and collecting the powers of
$a$, we may rewrite this equation as
\beqq
N(a,x)=0,
\eeqq
where
\beqq
N(a,x)
=
(x^2+3x+2)a^2
-
(x^2+6x+6)a
+
x+2.
\eeqq

Now fix $x>0$. The polynomial $N(a,x)$ is quadratic in $a$ with
positive leading coefficient. Since
\beqq
N(0,x)=x+2>0
\qquad \text{and} \qquad 
N(1,x)=-2(x+1)<0,
\eeqq
it has a root in $(0,1)$. In addition,
\beqq
N(1,x)<0
\qquad  \text{and} \qquad 
\lim_{a\to+\infty}N(a,x)=+\infty,
\eeqq
so it has a root in $(1,\infty)$. Since $N(a,x)$ is quadratic, these are its only two roots. Hence
$N(a,x)=0$ has exactly one root in $(1,\infty)$, which we denote by
$a_+(x)$.

Since $a_+(x)$ is the larger root and the leading coefficient of
$N(a,x)$ is positive, we have
\beqq
N_a(a_+(x),x)>0.
\eeqq
The implicit function theorem  shows that $a_+(x)$ is
continuously differentiable, with
\beqq
a_+'(x)
=
-\frac{N_x(a_+(x),x)}
{N_a(a_+(x),x)}.
\eeqq

Now define
\beqq
H(x) = G(a_+(x), x).
\eeqq
Applying the chain rule yields
\beqq
H'(x)
=
G_x(a_+(x), x)
-
G_a(a_+(x), x)
\frac{N_x(a_+(x), x)}
{N_a(a_+(x), x)},
\eeqq
which can be rewritten as
\beqq
H'(x)
=
\frac{
G_x(a_+(x), x) N_a(a_+(x), x)
-
G_a(a_+(x), x) N_x(a_+(x), x)
}
{N_a(a_+(x), x)}.
\eeqq

The relevant partial derivatives are given by
\beq\label{eq1}
G_a(a_+(x), x)
=
x-\frac{1+x}{a_+(x)(1+x)-1}
+\frac1{a_+(x)-1},
\eeq
\beqq
G_x(a_+(x), x)
=
a_+(x)-\frac1{1+x}
-\frac{a_+(x)}{a_+(x)(1+x)-1},
\eeqq
\beqq
N_a(a_+(x), x)
=
2(x^2+3x+2)a_+(x)-(x^2+6x+6),
\eeqq
and
\beqq
N_x(a_+(x), x)
=
(2x+3)a_+(x)^2-(2x+6)a_+(x)+1.
\eeqq

We now compute the numerator of \( H'(x) \) explicitly. 

Note that \( N(a,x) \) can be expressed as
\beqq
N(a,x)
=
(x+2)(a-1)[a(1+x)-1]-2a(1+x).
\eeqq
Since \( N(a_+(x), x) = 0 \), we have
\beq\label{eq6}
(x+2)(a_+(x)-1)[a_+(x)(1+x)-1]
=
2a_+(x)(1+x).
\eeq
This implies the useful identities
\beq\label{eq3}
\frac{1}{
(a_+(x)-1)[a_+(x)(1+x)-1]
}
=
\frac{x+2}{2a_+(x)(1+x)}
\eeq
and
\beq\label{eq4}
\frac{1}{a_+(x)(1+x)-1}
=
\frac{(x+2)(a_+(x)-1)}
{2a_+(x)(1+x)}.
\eeq

Using \eqref{eq1} and \eqref{eq3}, we simplify \( G_a \) as follows:
\beqq
\begin{aligned}
G_a(a_+(x), x)
&=
x-\frac{1+x}{a_+(x)(1+x)-1}
+\frac1{a_+(x)-1}
\\
&=
x+
\frac{
-(1+x)(a_+(x)-1)+a_+(x)(1+x)-1
}{
(a_+(x)-1)[a_+(x)(1+x)-1]
}
\\
&=
x+
\frac{x}{
(a_+(x)-1)[a_+(x)(1+x)-1]
}
\\
&=
x+\frac{x(x+2)}{2a_+(x)(1+x)}
\\
&=
\frac{
x[2a_+(x)(1+x)+x+2]
}{
2a_+(x)(1+x)
}.
\end{aligned}
\eeqq

Multiplying \eqref{eq4} by $a_+(x) $, we obtain
\beqq
\frac{a_+(x)}{a_+(x)(1+x)-1}
=
\frac{(x+2)(a_+(x)-1)}
{2(1+x)}.
\eeqq
Consequently, $ G_x(a_+(x), x) $ simplifies to
\beqq
\begin{aligned}
G_x(a_+(x), x)
&=
a_+(x)-\frac1{1+x}
-\frac{a_+(x)}{a_+(x)(1+x)-1}
\\
&=
a_+(x)-\frac1{1+x}
-\frac{(x+2)(a_+(x)-1)}{2(1+x)}
\\
&=
\frac{
2a_+(x)(1+x)-2-(x+2)(a_+(x)-1)
}{
2(1+x)
}
\\
&=
\frac{
2a_+(x)+2xa_+(x)-2
-xa_+(x)-2a_+(x)+x+2
}{
2(1+x)
}
\\
&=
\frac{x[a_+(x)+1]}{2(1+x)}.
\end{aligned}
\eeqq

Substituting the simplified forms of \( G_x \) and \( G_a \) into the numerator of \( H'(x) \), we get
\begin{align*}
&
G_x(a_+(x), x) N_a(a_+(x), x)
-
G_a(a_+(x), x) N_x(a_+(x), x)
\\
&=
\frac{x}{2a_+(x)(1+x)}
\Bigl\{
a_+(x)[a_+(x)+1]N_a(a_+(x), x)
-
[2a_+(x)(1+x)+x+2]N_x(a_+(x), x)
\Bigr\}.
\end{align*}

Replacing \( N_a \) and \( N_x \) with their explicit formulas, expanding, and collecting terms in powers of \( a_+(x) \), we obtain
\begin{align*}
&
a_+(x)[a_+(x)+1]N_a(a_+(x), x)
-
[2a_+(x)(1+x)+x+2]N_x(a_+(x), x)
\\
&=
-2(1+x)^2a_+(x)^3
+
(3x^2+9x+4)a_+(x)^2
+
(x^2+2x+4)a_+(x)
-
(x+2).
\end{align*}

By applying Lemma~\ref{lem2} and using the relation \( N(a_+(x), x)=0 \), the cubic polynomial above can be factored as
\begin{align*}
&
-2(1+x)^2a_+(x)^3
+
(3x^2+9x+4)a_+(x)^2
+
(x^2+2x+4)a_+(x)
-
(x+2)
\\
&=
\frac{
x[a_+(x)(1+x)-1]
\left[
a_+(x)(2x^3+8x^2+13x+8)-x-2
\right]
}{
(1+x)^2(x+2)
}.
\end{align*}

Therefore, the numerator of \( H'(x) \) becomes
\begin{align*}
&
G_x(a_+(x), x) N_a(a_+(x), x)
-
G_a(a_+(x), x) N_x(a_+(x), x)
\\
&=
\frac{
x^2[a_+(x)(1+x)-1]
\left[
a_+(x)(2x^3+8x^2+13x+8)-x-2
\right]
}{
2a_+(x)(1+x)^3(x+2)
}.
\end{align*}

Using the identity \eqref{eq6} once more, we eliminate the factor \( [a_+(x)(1+x)-1] \) and simplify to
\beqq
G_x(a_+(x), x) N_a(a_+(x), x)
-
G_a(a_+(x), x) N_x(a_+(x), x)
=
\frac{
x^2
\left[
a_+(x)(2x^3+8x^2+13x+8)-x-2
\right]
}{
(a_+(x)-1)(1+x)^2(x+2)^2
}.
\eeqq

Finally, dividing by \( N_a(a_+(x), x) \) gives the desired expression for \( H'(x) \):
\beqq
H'(x)
=
\frac{
x^2\left[
a_+(x)(2x^3+8x^2+13x+8)-x-2
\right]
}{
(a_+(x)-1)(1+x)^2(x+2)^2 N_a(a_+(x), x)
}.
\eeqq

Since $a_+(x)>1$, we have
\begin{align*}
&a_+(x)(2x^3+8x^2+13x+8)-x-2\\
&\qquad>
2x^3+8x^2+13x+8-x-2\\
&\qquad=
2x^3+8x^2+12x+6>0.
\end{align*}
Since $N_a(a_+(x),x)>0$ as well, we conclude that
\beqq
H'(x)>0.
\eeqq for $x>0$.

The quadratic formula gives the following expression for the larger root
of $N(a,x)=0$:
\beqq
a_+(x)
=
\frac{
x^2+6x+6+
\sqrt{(x^2+6x+6)^2-4(x^2+3x+2)(x+2)}
}
{2(x^2+3x+2)}.
\eeqq
It follows that
\beqq
\lim_{x\to0^+}a_+(x)
=
\frac{6+\sqrt{20}}4
=
\frac{3+\sqrt5}{2}
=:a_0
\eeqq
and
\begin{align*}
\lim_{x\to0^+}H(x)
&=
\lim_{x\to0^+}
\left[
a_+(x)x
-
\log
\frac{(1+x)\bigl(a_+(x)(1+x)-1\bigr)}
{a_+(x)-1}
\right]\\
&=0.
\end{align*}
Combining this limit with $H'(x)>0$ for $x>0$, we obtain
\beqq
H(x)>0
\qquad(x>0).
\eeqq
Hence $G(a,x)=0$ and $N(a,x)=0$ have no common solution for
$a>1$ and $x>0$, which implies that  $P$ has no critical point in
$D^\circ$.

We next examine the boundary behavior. For each fixed $a\geq1$,
\beqq
\lim_{x\to0^+}
\frac{1-(1+x)e^{-ax}}{x}
=
a-1.
\eeqq
Hence,
\beqq
\lim_{x\to0^+}P(a,x)
=
B(a):=2a(a-1)e^{1-a}.
\eeqq

For $a>1$, we have $B(a)>0$. Moreover,
\beqq
\frac{B'(a)}{B(a)}
=
\frac1a+\frac1{a-1}-1,
\eeqq
so the critical points of $B$ on $(1,\infty)$ are determined by
\beqq
a^2-3a+1=0.
\eeqq
The unique root larger than $1$ is
\beqq
a_0=\frac{3+\sqrt5}{2}.
\eeqq
In addition,
\beqq
B(1)=0,
\qquad
B(a)>0\quad(a>1) \qquad \text{and} \qquad 
\lim_{a\to+\infty}B(a)=0.
\eeqq
It follows that $B$ attains its maximum on $[1,\infty)$ at
$a_0$. Therefore,
\be\label{eq-2}\max_{a\geq1}B(a)
=
B(a_0)
=
2\left(\frac{1+\sqrt5}{2}\right)^3
e^{-\frac{1+\sqrt5}{2}}
=
\left(4+2\sqrt5\right)e^{-\frac{1+\sqrt5}{2}}.
\ee

We now establish the uniform boundary estimate
\beqq
\limsup_{x\to0^+}\sup_{a\geq1}P(a,x)
\leq
\left(4+2\sqrt5\right)e^{-\frac{1+\sqrt5}{2}}.
\eeqq
 For every $a\geq1$ and $x>0$, we have
\beq\label{eq-3}
1-(1+x)e^{-ax}
\leq
1-e^{-ax}
\leq
ax.
\eeq
The first inequality follows from
$$(1+x)e^{-ax}\geq e^{-ax},$$ and the second from
$$1-e^{-t}\leq t$$ for $t\geq0$.

Hence, for $0<x\leq1$,
\beq\label{eq-4}
P(a,x)
\leq
a^2e^{1-a}(x+2)
\leq
3a^2e^{1-a}.
\eeq

For a fixed $\varepsilon>0$, since
\beqq
\lim_{a\to+\infty}3a^2e^{1-a}=0,
\eeqq
we may select $A>1$ such that, for all $a\geq A$,
\beqq
3a^2e^{1-a}<\varepsilon.
\eeqq
Then \eqref{eq-4} gives
\beq\label{eq-5}
P(a,x)<\varepsilon
\qquad(a\geq A,\ 0<x\leq1).
\eeq

We next show that
\beqq
P(a,x)\longrightarrow B(a)
\qquad(x\to0^+)
\eeqq
uniformly for $a\in[1,A]$. The third derivative of $e^{-ax}$ with respect to $x$ is
$-a^3e^{-ax}$, and it is uniformly bounded for $a\in[1,A]$ and
$0\leq x\leq1$. Therefore, Taylor's formula gives
\beqq
e^{-ax}
=
1-ax+\frac{a^2x^2}{2}+O(x^3)
\qquad(x\to0^+),
\eeqq
uniformly for $a\in[1,A]$. Substituting this expansion, we obtain
\begin{align*}
1-(1+x)e^{-ax}
&=
1-(1+x)
\left(
1-ax+\frac{a^2x^2}{2}+O(x^3)
\right)\\
&=
(a-1)x+
\left(a-\frac{a^2}{2}\right)x^2
+
O(x^3),
\end{align*}
uniformly for $a\in[1,A]$. After dividing by $x$, we find
\beqq
\frac{1-(1+x)e^{-ax}}{x}
=
a-1+
\left(a-\frac{a^2}{2}\right)x
+
O(x^2).
\eeqq
Hence,
\beq\label{eq-6}
\sup_{1\leq a\leq A}
\left|
\frac{1-(1+x)e^{-ax}}{x}-(a-1)
\right|
\longrightarrow0
\qquad(x\to0^+).
\eeq

We shall also need the estimate
\beq\label{eq-7}
ae^{1-a}\leq1
\qquad(a\geq1),
\eeq
because the function
$a\mapsto ae^{1-a}$ is decreasing on $[1,\infty)$.

Using the definitions of $P$ and $B$, we have
\begin{align*}
P(a,x)-B(a)
&=
ae^{1-a}
\left\{
(x+2)\frac{\left[1-(1+x)e^{-ax}\right]}{x}
-2(a-1)
\right\}\\
&=
ae^{1-a}
\left\{
(x+2)
\left[
\frac{1-(1+x)e^{-ax}}{x}-(a-1)
\right]
+x(a-1)
\right\}.
\end{align*}
By \eqref{eq-7}, for $1\leq a\leq A$ and $0<x\leq1$,
\beqq
|P(a,x)-B(a)|
\leq
3
\left|
\frac{1-(1+x)e^{-ax}}{x}-(a-1)
\right|
+x(A-1).
\eeqq
Taking the supremum over $a\in[1,A]$ and applying \eqref{eq-6}, we
obtain
\beq\label{eq-8}
\sup_{1\leq a\leq A}
|P(a,x)-B(a)|
\longrightarrow0
\qquad(x\to0^+).
\eeq

Therefore, \eqref{eq-8} allows us to choose $0<\delta\leq1$
such that
\beqq
|P(a,x)-B(a)|<\varepsilon
\qquad
(1\leq a\leq A,\ 0<x\leq\delta).
\eeqq
Combining this with \eqref{eq-2}, we get
\beqq
P(a,x)
<
B(a)+\varepsilon
\leq
(4+2\sqrt5)e^{-\frac{1+\sqrt5}{2}}
+\varepsilon
\eeqq
for $1\leq a\leq A$ and $0<x\leq\delta$.

On the other hand, if $a\geq A$ and $0<x\leq\delta$, then
\eqref{eq-5} gives
$$P(a,x)<\varepsilon.$$ Combining these two cases, we obtain
\beqq
\sup_{a\geq1}P(a,x)
\leq
(4+2\sqrt5)e^{-\frac{1+\sqrt5}{2}}
+\varepsilon
\qquad(0<x\leq\delta).
\eeqq
Since this holds for every $\varepsilon>0$, we conclude that
\beqq
\limsup_{x\to0^+}\sup_{a\geq1}P(a,x)
\leq
(4+2\sqrt5)e^{-\frac{1+\sqrt5}{2}}.
\eeqq

Next, combining \eqref{eq-7} with
$
1-(1+x)e^{-ax}<1,
$
we have
\beqq
P(a,x)
\leq
\frac{x+2}{x}
=
1+\frac2x.
\eeqq
Therefore,
\beqq
\limsup_{x\to+\infty}
\sup_{a\geq1}P(a,x)
\leq1<B(2)\leq(4+2\sqrt5)e^{-\frac{1+\sqrt5}{2}}.
\eeqq

Finally, \eqref{eq-3} yields
\begin{align*}
\frac{(x+2)}{x}
\left[1-(1+x)e^{-ax}\right]
&=
1-(1+x)e^{-ax}+
\frac2x
\left[1-(1+x)e^{-ax}\right]\\
&\leq
1+2a.
\end{align*}
Hence,
\beq\label{eq-9}
P(a,x)
\leq
ae^{1-a}(1+2a)
\qquad(x>0).
\eeq
Taking the supremum over $x>0$ in \eqref{eq-9} gives
\beqq
0
\leq
\sup_{x>0}P(a,x)
\leq
ae^{1-a}(1+2a).
\eeqq
Since
\beqq
\lim_{a\to+\infty}
ae^{1-a}(1+2a)
=
0,
\eeqq
we conclude that
\beqq
\lim_{a\to+\infty}
\sup_{x>0}P(a,x)
=
0.
\eeqq

We are now ready to prove the upper bound. Let \(\varepsilon>0\).
The preceding boundary estimates
allow us to choose \(0<\delta<R\) and \(A>1\) such that
\beq\label{eq-100}
P(a,x)
\leq
(4+2\sqrt5)e^{-\frac{1+\sqrt5}{2}}+\varepsilon
\eeq
for every
\beqq
(a,x)\in
\Omega_{1}
\cup
\Omega_{2}
\cup
\Omega_{3},
\eeqq
where $\Omega_{1}=[1,\infty)\times(0,\delta]$, $\Omega_{2}=[1,\infty)\times[R,\infty)$ and $\Omega_{3}=[A,\infty)\times(0,\infty)$.

Consider the compact rectangle
$
\mathcal{K}=[1,A]\times[\delta,R].
$
Assume, for contradiction, that
\beqq
P(a,x)
>
\left(4+2\sqrt5\right)e^{-\frac{1+\sqrt5}{2}}
+\varepsilon
\eeqq
at some point of $\mathcal{K}$. Because $P$ is continuous on $\mathcal{K}$, it attains
its maximum at some point $(a_*,x_*)\in \mathcal{K}$, and
\beqq
P(a_*,x_*)
>
\left(4+2\sqrt5\right)e^{-\frac{1+\sqrt5}{2}}
+\varepsilon.
\eeqq
We will show that $(a_*,x_*)\notin \partial \mathcal{K}$.
By \eqref{eq-100}, we have
\beqq
(a_*,x_*)
\notin
\Omega_{4}
\cup
\Omega_{5}
\cup
\Omega_{6},
\eeqq where $\Omega_{4}=[1,A]\times\{\delta\}$, $\Omega_{5}=[1,A]\times\{R\}$ and $\Omega_{6}=\{A\}\times[\delta,R]$.
It remains to show that
\beqq
(a_*,x_*)\notin\{1\}\times[\delta,R].
\eeqq
Indeed, for each fixed $x>0$, the formula
\beqq
\ell_a(a,x)
=
\frac1a-1+
\frac{x(1+x)e^{-ax}}
{1-(1+x)e^{-ax}}
\eeqq
extends continuously to $a=1$. Since both sides of \eqref{eq-0} also extend continuously to the boundary, and
\beqq
\ell_a(1,x)
=
\frac{x(1+x)e^{-x}}
{1-(1+x)e^{-x}}
>0,
\eeqq
the function $a\mapsto P(a,x)$ is strictly increasing in a right-hand
neighborhood of $a=1$. Hence the maximum of $P$ on
$\mathcal{K}$ cannot occur on the side $\{1\}\times[\delta,R]$.

Therefore,
\beqq
(a_*,x_*)\in\operatorname{int}\mathcal{K}
=
(1,A)\times(\delta,R)
\subset D^\circ.
\eeqq
Since $(a_*,x_*)$ is an interior maximum point of $P$ and $P$ is
smooth in $D^\circ$, Fermat's theorem gives
\beqq
P_a(a_*,x_*)=P_x(a_*,x_*)=0.
\eeqq
Thus $(a_*,x_*)$ is a critical point of $P$ in $D^\circ$,
contradicting the fact proved above that $P$ has no critical point in
$D^\circ$. Therefore,
\beqq
P(a,x)
\leq
(4+2\sqrt5)e^{-\frac{\left(1+\sqrt5\right)}{2}}
+\varepsilon
\qquad(a\geq1,\ x>0).
\eeqq
Since the preceding inequality holds for every $\varepsilon>0$, we obtain
\beqq
P(a,x)
\leq
(4+2\sqrt5)e^{-\frac{\left(1+\sqrt5\right)}{2}}
\qquad(a\geq1,\ x>0).
\eeqq
Taking the supremum over the whole domain, we conclude that
\beqq
\sup_{a\geq1,\ x>0}P(a,x)
\leq
\left(4+2\sqrt5\right)e^{-\frac{\left(1+\sqrt5\right)}{2}}.
\eeqq

To prove the reverse inequality, choose
\beqq
a=a_0=\frac{3+\sqrt5}{2}.
\eeqq
For this choice,
\beqq
\lim_{x\to0^+}
\frac{1-(1+x)e^{-a_0x}}{x}
=
a_0-1.
\eeqq
Therefore,
$$\lim_{x\to0^+}P(a_0,x)=
2a_0(a_0-1)e^{1-a_0}=
\left(4+2\sqrt5\right)e^{-\frac{\left(1+\sqrt5\right)}{2}}.
$$
It follows that
\beqq
\sup_{a\geq1,\ x>0}P(a,x)
\geq
\left(4+2\sqrt5\right)e^{-\frac{\left(1+\sqrt5\right)}{2}}.
\eeqq
Combining the upper and lower bounds, we obtain
\beqq
\sup_{a\geq1,\ x>0}P(a,x)
=
\left(4+2\sqrt5\right)e^{-\frac{\left(1+\sqrt5\right)}{2}}.
\eeqq
The proof of this lemma is complete.
\end{proof}



Now we come to prove Theorem \ref{CL-lem-2}.
 Since $\mathbb{D}$ is a simply connected domain, it follows that $f$ admits the canonical decomposition
$$
f = h + \overline{g},
$$
where $h$ and $g$ are analytic in $\mathbb{D}$ with $g(0)=0$. If $\|f\|_{\mathscr{B}_{s}}=0$, then (\ref{eq-cl-10})
holds. Without loss of generality, we assume that  $\|f\|_{\mathscr{B}_{s}}=1$ and $\mathcal{B}_{f}(z_{1})\leq\mathcal{B}_{f}(z_{2}).$
For $z\in\mathbb{D}$, let 

$$\phi(z)=\frac{z_{2}-z}{1-\overline{z_{2}}z},~~w=\phi^{-1}(z_{1})\,\,\,~~~\mbox{and}\,\,\,~~~F(z)=\mathbf{H}(z)+\overline{\mathbf{G}(z)},$$
where $\mathbf{H}=h\circ\phi$ and $\mathbf{G}=g\circ\phi$. Then
$$\|F\|_{\mathscr{B}_{s}}=\|f\|_{\mathscr{B}_{s}}=1$$ and 
\be\label{eq-po-1}\rho(z_{1},z_{2})=\rho\left(\phi^{-1}(z_{1}),\phi^{-1}(z_{2})\right)=\rho(w,0)=|w|.\ee
Since $$\phi(0)=z_{2}~~\mbox{and}~~|\phi'(w)|=\frac{1-|\phi(w)|^{2}}{1-|w|^{2}}=\frac{1-|z_{1}|^{2}}{1-|w|^{2}},$$
we see that 
\be\label{eq-cj-11}
\mathcal{B}_{F}(0)=\mathcal{B}_{f}(\phi(0))=\mathcal{B}_{f}(z_{2})
\ee
and 
\be\label{eq-cj-12}
\mathcal{B}_{F}(w)=\left(1-|w|^{2}\right)\Lambda_{f}(\phi(w))|\phi'(w)|=\mathcal{B}_{f}(z_{1}).
\ee

Since $F$ is a locally univalent harmonic function on $\mathbb{D}$, we have
$\mathcal{B}_{F}(0)>0$.

For $z\in\mathbb{D}$, let $$F_{\theta_{0}}(z)=\mathbf{H}(z)+e^{i\theta_{0}}\mathbf{G}(z),$$
where $\theta_{0}\in[0,2\pi]$ is a real number such that $|F_{\theta_{0}}'(0)|=\mathcal{B}_{F}(0)$.
Then $F_{\theta_{0}}$ is a locally univalent analytic function on $\mathbb{D}$,
since $F$ is a locally univalent harmonic function on $\mathbb{D}$.
Also, $\|F_{\theta_{0}}\|_{\mathscr{B}_{s}}\leq1$.
Let $F_{\theta_{0}}'(0)=\alpha e^{i\eta}$, where $\alpha=\mathcal{B}_{F}(0)$.
An application of Lemma \ref{CL-lem-1} to $e^{-i\eta}F_{\theta_{0}}(z)$ gives that, for all $z\in\mathbb{D}$, 

\be\label{eq-cj-13}
\mathcal{B}_{F}(z)\geq \mathcal{B}_{F_{\theta_{0}}}(z)\geq(1+m(\alpha))\frac{1+|z|}{1-|z|}e^{\left[1-(1+m(\alpha))\frac{1+|z|}{1-|z|}\right]},
\ee
where $m(\alpha)$ is the same as in (\ref{eq-CJ-12}).
It follows from (\ref{eq-CJ-12}), (\ref{eq-cj-11}), (\ref{eq-cj-12}) and (\ref{eq-cj-13}) that 
\beq\label{eq-cj-144}
\nonumber \mathcal{B}_{f}(z_{2})-\mathcal{B}_{f}(z_{1})&=&\mathcal{B}_{F}(0)-\mathcal{B}_{F}(w)\\
&\leq&\alpha-(1+m(\alpha))\frac{1+|w|}{1-|w|}e^{\left[1-(1+m(\alpha))\frac{1+|w|}{1-|w|}\right]}\nonumber\\
&=&\alpha\left[1-\frac{1+|w|}{1-|w|}e^{-\frac{2(1+m(\alpha))|w|}{1-|w|}}\right].
\eeq
If $w=0$, then $z_1=z_2$, and \eqref{eq-cl-10} holds trivially.
Hence, in what follows, we assume that
$
0<|w|<1.
$

Now set
\beqq
a=1+m(\alpha) \quad \text{and} \quad x=\frac{2|w|}{1-|w|}.
\eeqq
Since $m(\alpha)\geq0$ and $0<|w|<1$, we have
\beqq
a\geq1,\qquad \text{and} \qquad x>0.
\eeqq
Moreover,
\[
\alpha=ae^{1-a},\qquad
\frac{1+|w|}{1-|w|}=1+x,\qquad
\frac{2(1+m(\alpha))|w|}{1-|w|}=ax
\qquad\text{and} \qquad
|w|=\frac{x}{x+2}.
\]
Hence, 
\beqq
\alpha\left[
1-\frac{1+|w|}{1-|w|}
e^{-\frac{2(1+m(\alpha))|w|}{1-|w|}}
\right]
&=&ae^{1-a}\left[1-(1+x)e^{-ax}\right]\\
&=&|w|\,ae^{1-a}\frac{x+2}{x}
\left[1-(1+x)e^{-ax}\right]\\
&=&|w|P(a,x)\\
&\leq&C_0|w|,
\eeqq
where the last inequality follows from Lemma \ref{lem1}.  Together with \eqref{eq-po-1} and \eqref{eq-cj-144}, yields that
\beqq
\mathcal{B}_{f}(z_{2})-\mathcal{B}_{f}(z_{1})\le C_0 \rho(z_{1},z_{2}).
\eeqq

It remains to prove the sharpness of the constant $C_0$. Set
\[
m_0=\frac{1+\sqrt5}{2},
\qquad
\alpha_0=(1+m_0)e^{-m_0},
\]
and define
\[
f_*(z)
=
\int_0^z
\frac{\alpha_0}{(1-\xi)^2}
e^{\left[
-\frac{2(1+m_0)\xi}{1-\xi}
\right]}
\,d\xi,
\qquad z\in\mathbb D.
\]
Then
\[
f_*'(z)
=
\frac{\alpha_0}{(1-z)^2}
e^{\left[
-\frac{2(1+m_0)z}{1-z}
\right]}.
\]
Since $f_*'(z)\neq0$ for every $z\in\mathbb D$, the function $f_*$
is locally univalent in $\mathbb D$.

We regard $f_*$ as a harmonic mapping whose co-analytic part is
identically zero. Thus,
\[
\Lambda_{f_*}(z)=|f_*'(z)|
\]
and
\[
\mathcal B_{f_*}(z)
=
(1-|z|^2)|f_*'(z)|.
\]
We next show that
\[
\|f_*\|_{\mathscr B_s}=1.
\]

For every $z\in\mathbb D$, using
\[
\alpha_0=(1+m_0)e^{-m_0}
\]
and the identity
\[
2\operatorname{Re}\left(\frac{z}{1-z}\right)
=
\frac{1-|z|^2}{|1-z|^2}-1,
\]
we obtain
\[
\begin{aligned}
\mathcal B_{f_*}(z)
&=
(1-|z|^2)|f_*'(z)|\\
&=
\frac{(1-|z|^2)}{|1-z|^2}
\alpha_0
e^{\left[
-2(1+m_0)\operatorname{Re}\left(\frac{z}{1-z}\right)
\right]}\\
&=
\frac{(1-|z|^2)}{|1-z|^2}
(1+m_0)e^{-m_0}
e^{\left[
-(1+m_0)
\left(
\frac{1-|z|^2}{|1-z|^2}-1
\right)
\right]}\\
&=
e(1+m_0)
\frac{(1-|z|^2)}{|1-z|^2}
e^{\left[
-(1+m_0)\frac{(1-|z|^2)}{|1-z|^2}
\right]}\\
&\leq1.
\end{aligned}
\]
The final inequality follows from
$
te^{1-t}\le1
$
for any $t>0$ with equality only for $t=1$.
Consequently,
\[
\|f_*\|_{\mathscr B_s}
=
\sup_{z\in\mathbb D}\mathcal B_{f_*}(z)
\leq1.
\]

On the other hand, as $z$ varies over $(-1,1)$, the quantity
\[
\frac{1-|z|^2}{|1-z|^2}
=
\frac{1+z}{1-z}
\]
ranges over $(0,\infty)$. Hence, we may choose $z\in(-1,1)$ such
that
\[
(1+m_0)\frac{(1-|z|^2)}{|1-z|^2}=1.
\]
For this choice of $z$, equality holds in the preceding estimate.
Therefore,
\[
\|f_*\|_{\mathscr B_s}=1.
\]

Now take
\[
z_1=0,\qquad z_2=r,
\qquad 0<r<1.
\]
Then
\[
\rho(z_1,z_2)=\rho(0,r)=r,
\]
and
\[
\mathcal B_{f_*}(0)
=
|f_*'(0)|
=
\alpha_0.
\]
Moreover,
\[
\mathcal B_{f_*}(r)
=
(1-r^2)|f_*'(r)|
=
\alpha_0\left(\frac{1+r}{1-r}\right)
e^{\left[
-\frac{2(1+m_0)r}{1-r}
\right]}.
\]
Since $m_0>0$, for all sufficiently small $r>0$,
\[
\left(\frac{1+r}{1-r}\right)
e^{\left[
-\frac{2(1+m_0)r}{1-r}
\right]}
<1.
\]
Thus,
\[
\mathcal B_{f_*}(0)>\mathcal B_{f_*}(r)
\]
for all sufficiently small $r>0$, and hence
\[
\left|
\mathcal B_{f_*}(0)-\mathcal B_{f_*}(r)
\right|
=
\mathcal B_{f_*}(0)-\mathcal B_{f_*}(r).
\]
Therefore,
\[
\begin{aligned}
\frac{
\left|
\mathcal B_{f_*}(0)-\mathcal B_{f_*}(r)
\right|
}{
\|f_*\|_{\mathscr B_s}\rho(0,r)
}
&=
\frac{\alpha_0}{r}
\left\{
1-
\left(\frac{1+r}{1-r}\right)
e^{\left[
-\frac{2(1+m_0)r}{1-r}
\right]}
\right\}.
\end{aligned}
\]
As $r\to0^+$,
\[
\left(\frac{1+r}{1-r}\right)
e^{\left[
-\frac{2(1+m_0)r}{1-r}
\right]}
=
1-2m_0r+O(r^2).
\]
It follows that
\[
\begin{aligned}
\lim_{r\to0^+}
\frac{
\left|
\mathcal B_{f_*}(0)-\mathcal B_{f_*}(r)
\right|
}{
\|f_*\|_{\mathscr B_s}\rho(0,r)
}
&=
2m_0\alpha_0\\
&=
2m_0(1+m_0)e^{-m_0}\\
&=
(4+2\sqrt5)
e^{-\frac{1+\sqrt5}{2}}\\
&=
C_0.
\end{aligned}
\]
Therefore, no constant smaller than $C_0$ can satisfy
\eqref{eq-cl-10} for all locally univalent mappings in
$\mathscr B_h$. Hence, the constant $C_0$ is sharp.
The proof of this theorem is complete.
\qed



\subsection*{Proof of Theorem \ref{thm-1}} In preparation for the proof of this theorem, we first prove some required lemmas.

\begin{lem}\label{lem-2.3}
Let $f$ be a harmonic $(K,K_{0})$-quasiregular mapping in $\mathbb{D}$. Then 
\begin{enumerate}
\item
$J_f(z)\geq 0$ in $\mathbb{D}$;
\item
$\Lambda_f^2(z) \le K J_f(z) + K_0$  in $\mathbb{D}$;
\item
$f\in\mathscr{B}_h$
 if and only if $f\in\mathscr{B}_h^*$. Moreover, we have
\beqq
    \|f\|_{\mathscr{B}_{h^*,s}}
    \leq
    \|f\|_{\mathscr{B}_{s}}
    \leq
    \sqrt{K}\,\|f\|_{\mathscr{B}_{h^*,s}}+ \sqrt{K_{0}}.
\eeqq
\end{enumerate}
\end{lem}

\begin{proof}
First, we will prove (2). The proof for (1) is similar.
Since $f$ is a harmonic $(K,K_0)$-quasiregular mapping, the functions $\Lambda_f$ and $J_f$ are continuous on $\mathbb{D}$, and the inequality
\beqq
    \Lambda_f^2 \le K J_f + K_0
\eeqq
holds almost everywhere in $\mathbb{D}$; in fact, it holds everywhere in $\mathbb{D}$. Indeed, if it failed at some point, the continuity of the function
\[
\Lambda_f^2 - K J_f - K_0
\]
would imply the existence of a disk of positive measure on which the inequality is violated, a contradiction.

Next, we prove (3).
Suppose $f\in\mathscr{B}_h$. Since $\sqrt{J_f(z)}\leq \Lambda_f(z)$, it immediately follows that
\beqq
    \sup_{z\in\mathbb{D}}(1-|z|^2)\sqrt{J_f(z)}
    \leq
    \sup_{z\in\mathbb{D}}(1-|z|^2)\Lambda_f(z)
    =
    \|f\|_{\mathscr{B}_{s}}.
\eeqq
This implies that $f\in\mathscr{B}_h^*$ and 
$\|f\|_{\mathscr{B}_{h^*,s}}\leq \|f\|_{\mathscr{B}_{s}}$.

Conversely, suppose $f\in\mathscr{B}_h^*$. 

Consequently,
\beqq
    \Lambda_f(\xi)
    \le
    \sqrt{K}\,\sqrt{J_f(\xi)} + \sqrt{K_0}, \qquad \xi \in \mathbb{D}.
\eeqq
Multiplying both sides by $1-|\xi|^2$ and taking the supremum over $\xi\in\mathbb{D}$, we obtain
\beqq
    \|f\|_{\mathscr{B}_{s}}
    \le
    \sqrt{K}\,\|f\|_{\mathscr{B}_{h^*,s}} + \sqrt{K_0}.
\eeqq
Thus the proof of the lemma is complete.
\end{proof}

\blem\label{lem-2.2}
Let \(f=h+\overline g\) be a harmonic \((K,K_{0})\)-quasiregular mapping in
\(\mathbb D\), where \(K\geq1\) and \(K_{0}\geq0\). If
\(f\in \mathscr{B}_h^*\), then
\beqq
\|g\|_{\mathscr{B}_{s}}
\leq
\|f\|_{\mathscr{B}_{h^*,s}}  \frac{\sqrt{K - 1}}{2}
+ \frac{\sqrt{K_0}}{2}
\eeqq
and
\beqq
\|h\|_{\mathscr{B}_{s}}
\leq
\|f\|_{\mathscr{B}_{h^*,s}}  \frac{\sqrt{K +3}}{2}
+ \frac{\sqrt{K_0}}{2}.
\eeqq
\elem

\begin{proof}
Set \(M = \|f\|_{\mathscr{B}_{h^*,s}}\). For \(z \in \mathbb D\), define
\beqq
a(z) = (1 - |z|^2)|h'(z)|
\qquad\text{and}\qquad
b(z) = (1 - |z|^2)|g'(z)|.
\eeqq
Since \(J_f(z) = |h'(z)|^2 - |g'(z)|^2\), we have
\beqq
(1 - |z|^2)^2 J_f(z) = a(z)^2 - b(z)^2.
\eeqq

By the assumption that \(f\) is a harmonic \((K, K_0)\)-quasiregular mapping, we obtain
\beqq
\Lambda_f^2(z) \leq K J_f(z) + K_0,
\qquad z \in \mathbb D.
\eeqq
Multiplying this inequality by \((1 - |z|^2)^2\), we get
\beqq
(a(z) + b(z))^2
\leq K\bigl(a(z)^2 - b(z)^2\bigr) + K_0 (1 - |z|^2)^2 
\leq K\bigl(a(z)^2 - b(z)^2\bigr) + K_0.
\eeqq
Equivalently,
\beqq
(K + 1) b(z)^2 + 2a(z)b(z)
\leq (K - 1) a(z)^2 + K_0.
\eeqq
Since $J_f(z)\geq 0$ in $\mathbb{D}$, by Lemma \ref{lem-2.3},
we obtain $$2a(z)b(z) \geq 2b(z)^2,$$ which combined with the above inequality implies that
\be\label{2.1}
(K +3) b(z)^2
\leq (K - 1) a(z)^2 + K_0.
\ee

On the other hand, by the definition of \(M\), we have
\beqq
(1 - |z|^2) \sqrt{J_f(z)} \leq M,
\eeqq
which implies that
\beqq
a(z)^2 - b(z)^2
= (1 - |z|^2)^2 J_f(z)
\leq M^2.
\eeqq
Consequently,
\begin{align}\label{eq-2.19}
a(z)^2 &\leq M^2 + b(z)^2.
\end{align}
Combining \eqref{2.1} and \eqref{eq-2.19} yields
\beqq
(K + 3) b(z)^2
\leq (K - 1)\bigl(M^2 + b(z)^2\bigr) + K_0.
\eeqq
Thus,
\beq\label{cq-chj1}
4 b(z)^2 \leq (K - 1) M^2 + K_0,
\eeq
and consequently,
\beqq
b(z)
\leq \frac{\sqrt{(K - 1) M^2 + K_0}}{2}
\leq M \frac{\sqrt{K - 1}}{2}
+ \frac{\sqrt{K_0}}{2}.
\eeqq
Taking the supremum over \(z \in \mathbb D\), we obtain
\beqq
\|g\|_{\mathscr{B}_{s}}=\sup_{z \in \mathbb D} b(z)
= \sup_{z \in \mathbb D} \mathcal{B}_{g}(z)
\leq M \frac{\sqrt{K - 1}}{2}
+ \frac{\sqrt{K_0}}{2}.
\eeqq

On the other hand, by \eqref{eq-2.19} and \eqref{cq-chj1}, we obtain
\beqq
4a(z)^2 \leq 4 M^2 + 4b(z)^2 \leq 4M^2+(K - 1) M^2 + K_0 =(K+3)M^2+K_0,
\eeqq
which implies that
\beqq
a(z)
\leq \frac{\sqrt{(K +3) M^2 + K_0}}{2}
\leq M \frac{\sqrt{K +3}}{2}
+ \frac{\sqrt{K_0}}{2}.
\eeqq
Therefore,
\beqq
\|h\|_{\mathscr{B}_{s}}
= \sup_{z \in \mathbb D} a(z)
= \sup_{z \in \mathbb D} \mathcal{B}_{h}(z)
\leq M \frac{\sqrt{K +3}}{2}
+ \frac{\sqrt{K_0}}{2}.
\eeqq
The proof of this lemma is complete.
\end{proof}



We now proceed to prove Theorem \ref{thm-1}.
Fix \(z,w\in\mathbb D\). We may assume that \(z\neq w\). Let
\beqq
\varphi_w(\xi)=\frac{w-\xi}{1-\overline w\,\xi},
\qquad \text{and} \qquad
\zeta=\varphi_w(z).
\eeqq
Then
\beqq
|\zeta|=\rho(z,w).
\eeqq
Set
\beqq
\psi=f\circ\varphi_w=H+\overline G,
\qquad
H=h\circ\varphi_w,\quad \text{and} \quad G=g\circ\varphi_w .
\eeqq
Since $\varphi_w$ is a conformal automorphism of $\mathbb D$, the mapping
$\psi$ is also harmonic, locally univalent and sense-preserving in
$\mathbb D$.
By the chain rule,
\beqq
J_\psi(\xi)=J_f(\varphi_w(\xi))|\varphi_w'(\xi)|^2.
\eeqq
Moreover,
\beqq
1-|\varphi_w(\xi)|^2=(1-|\xi|^2)|\varphi_w'(\xi)|.
\eeqq
It follows that
\beqq
(1-|\xi|^2)\sqrt{J_\psi(\xi)}
=
(1-|\varphi_w(\xi)|^2)\sqrt{J_f(\varphi_w(\xi))}.
\eeqq
Taking \(\xi=\zeta\) and \(\xi=0\), respectively, we obtain
\beqq
(1-|\zeta|^2)\sqrt{J_\psi(\zeta)}
=
(1-|z|^2)\sqrt{J_f(z)}
\eeqq
and
\beqq
\sqrt{J_\psi(0)}
=
(1-|w|^2)\sqrt{J_f(w)}.
\eeqq
Thus it remains to estimate
\beqq
|U(\zeta)-U(0)|,
\eeqq
where
\beqq
U(\xi)=(1-|\xi|^2)\sqrt{J_\psi(\xi)}.
\eeqq

For \(\xi\in\mathbb D\), set
\beqq
a(\xi)=(1-|\xi|^2)|H'(\xi)|
\qquad \text{and} \qquad
b(\xi)=(1-|\xi|^2)|G'(\xi)|.
\eeqq
Since
\beqq
J_\psi(\xi)=|H'(\xi)|^2-|G'(\xi)|^2,
\eeqq
we have
\beqq
U(\xi)^2=a(\xi)^2-b(\xi)^2.
\eeqq

Since \(\psi\) is sense-preserving and locally univalent, we have
\beqq
J_\psi(\xi)=|H'(\xi)|^2-|G'(\xi)|^2>0,
\qquad \xi\in\mathbb D.
\eeqq
Hence \(H'(\xi)\neq0\) for every \(\xi\in\mathbb D\), which implies that $H$ is also locally univalent in $\mathbb{D}$.  Therefore,
\beqq
\omega(\xi)=\frac{G'(\xi)}{H'(\xi)}
\eeqq
is analytic in \(\mathbb D\) and satisfies
\[
|\omega(\xi)|<1,
\qquad \xi\in\mathbb D.
\]
For $\xi\in\mathbb D$,
let
\beqq
q(\xi)=\sqrt{1-|\omega(\xi)|^2}.
\eeqq
Since
\beqq
U(\xi)
=
(1-|\xi|^2)\sqrt{|H'(\xi)|^2-|G'(\xi)|^2},
\eeqq
and
\beqq
|G'(\xi)|=|\omega(\xi)|\,|H'(\xi)|,
\eeqq
we get
\beqq
U(\xi)
=
(1-|\xi|^2)|H'(\xi)|
\sqrt{1-|\omega(\xi)|^2}
=
a(\xi)q(\xi).
\eeqq
Let
\beqq
U(\xi)=a(\xi)-E(\xi)
~~~\mbox{and}~~~
E(\xi)=a(\xi)(1-q(\xi)).
\eeqq
Since
\beqq
1-q(\xi)
=
\frac{(1-q(\xi))(1+q(\xi))}{1+q(\xi)}
=
\frac{1-q(\xi)^2}{1+q(\xi)}
=
\frac{|\omega(\xi)|^2}{1+q(\xi)}
\eeqq
and
\beqq
b(\xi)=a(\xi)|\omega(\xi)|,
\eeqq
we obtain
\beqq
E(\xi)
=
a(\xi)(1-q(\xi))
=
a(\xi)\frac{|\omega(\xi)|^2}{1+q(\xi)}
=
b(\xi)\frac{|\omega(\xi)|}{1+q(\xi)}.
\eeqq
Thus
\beqq
E(\xi)=b(\xi)s(\xi)~~~\,\,
\mbox{and}~~~\,\,
s(\xi)=\frac{|\omega(\xi)|}{1+q(\xi)}.
\eeqq
Clearly,
\beqq
0\leq s(\xi)<1,
\qquad \xi\in\mathbb D.
\eeqq
Therefore,
\beq\label{3.1}
|U(\zeta)-U(0)|
\leq
|a(\zeta)-a(0)|+|E(\zeta)-E(0)|.
\eeq

We now estimate the two terms in \eqref{3.1} separately. 
First, since \(f\) is a harmonic \((K,K_{0})\)-quasiregular mapping and
\(f\in \mathscr{B}_h^*\), Lemma 2.1 implies that \(f\in \mathscr{B}_h\). Hence
\(h,g\in \mathscr{B}\). Moreover, by the M\"obius invariance of the
analytic Bloch semi-norm, we have
\beqq
\|H\|_{\mathscr{B}_s}
=
\|h\circ\varphi_w\|_{\mathscr{B}_s}
=
\|h\|_{\mathscr{B}_s}
\eeqq
and similarly $G\in \mathscr{B}$.

Applying Theorem \ref{CL-lem-2} to the analytic function \(H\), we obtain
\beq\label{chj0}
|a(\zeta)-a(0)|
=
\left|
(1-|\zeta|^2)|H'(\zeta)|-|H'(0)|
\right|
\leq
 C_{0}|\zeta|\,\|H\|_{\mathscr{B}_s}=
 C_{0}|\zeta| \|h\|_{\mathscr{B}_s},
\eeq
 where $C_{0}$ is the same as in Theorem \ref{CL-lem-2}.
Furthermore, it follows from Lemma \ref{lem-2.2} that
\beqq
\|h\|_{\mathscr{B}_s}
\leq
\frac{\sqrt{K+3}}{2}\|f\|_{\mathscr{B}_{h^*,s}}+\frac{\sqrt{K_{0}}}{2}.
\eeqq
Consequently,
\beq\label{3.2}
|a(\zeta)-a(0)|
\leq
C_{0}|\zeta|
\left(
\frac{\sqrt{K+3}}{2}\|f\|_{\mathscr{B}_{h^*,s}}+\frac{\sqrt{K_{0}}}{2}
\right).
\eeq

Next we estimate the second term in \eqref{3.1}. It follows from
$E(\xi)=b(\xi)s(\xi)$ that
\begin{align*}
|E(\zeta)-E(0)|=|b(\zeta)s(\zeta)-b(0)s(0)|\le
|b(\zeta)-b(0)|s(\zeta)
+
b(0)|s(\zeta)-s(0)|.
\end{align*}
Since \(0\leq s(\xi)<1\) for \(\xi\in\mathbb D\), we see that
\be\label{3.t3}
|E(\zeta)-E(0)|
\leq
|b(\zeta)-b(0)|
+
b(0)|s(\zeta)-s(0)|.
\ee
 By applying \eqref{eq-cj-3}  to $G$, we get
\beq\label{3.3}
|b(\zeta)-b(0)|
=
\left|
(1-|\zeta|^2)|G'(\zeta)|-|G'(0)|
\right|
\leq
\frac{3\sqrt{3}}{2}|\zeta|\,\|G\|_{\mathscr{B}_{s}}.
\eeq
It remains to estimate \(|s(\zeta)-s(0)|\).

Let
\beqq
S(t)=\frac{t}{1+\sqrt{1-t^2}},
\qquad
0\leq t<1.
\eeqq
Then
\beqq
s(\xi)=S(|\omega(\xi)|)
\eeqq
and
\beq\label{3}
S'(t)
=
\frac{1}
{\sqrt{1-t^2}\bigl(1+\sqrt{1-t^2}\bigr)}.
\eeq

Since the function $r\mapsto |\omega(r\zeta)|$ is absolutely continuous on $[0,1]$, we see that, for almost every \(r\in[0,1]\), the chain rule gives
\beqq
\frac{d}{dr}s(r\zeta)
=
S'(|\omega(r\zeta)|)
\frac{d}{dr}|\omega(r\zeta)|.
\eeqq
Consequently,
\beq\label{3.4}
\left|
\frac{d}{dr}s(r\zeta)
\right|
\leq
S'(|\omega(r\zeta)|)
\left|
\frac{d}{dr}|\omega(r\zeta)|
\right|.
\eeq
Moreover,
\beqq
\left|
\frac{d}{dr}|\omega(r\zeta)|
\right|
\leq
\left|
\frac{d}{dr}\omega(r\zeta)
\right|
=
|\omega'(r\zeta)|\,|\zeta|.
\eeqq
By Schwarz--Pick's lemma for $\omega$,
\beqq
|\omega'(r\zeta)|
\leq
\frac{1-|\omega(r\zeta)|^2}
{1-r^2|\zeta|^2}.
\eeqq
Thus
\beqq
\left|
\frac{d}{dr}|\omega(r\zeta)|
\right|
\leq
\frac{1-|\omega(r\zeta)|^2}
{1-r^2|\zeta|^2}
|\zeta|.
\eeqq
Combining \eqref{3}
 and \eqref{3.4} yields
\begin{align*}
\left|
\frac{d}{dr}s(r\zeta)
\right|
&\leq
\frac{
1-|\omega(r\zeta)|^2
}
{
\sqrt{1-|\omega(r\zeta)|^2}
\left(1+\sqrt{1-|\omega(r\zeta)|^2}\right)
}
\cdot
\frac{|\zeta|}
{1-r^2|\zeta|^2}\\ 
&\leq
\frac{\sqrt{1-|\omega(r\zeta)|^2}}
{1+\sqrt{1-|\omega(r\zeta)|^2}}
\cdot
\frac{|\zeta|}
{1-r^2|\zeta|^2} \\
&\leq
\frac{1}{2}\frac{|\zeta|}
{1-r^2|\zeta|^2}.
\end{align*}
Integrating from \(0\) to \(1\), we get
\beqq
|s(\zeta)-s(0)|
\leq
\int_0^1\left|\frac{d}{dr}s(r\zeta)\right|\,dr
\leq
\int_0^1\frac{1}{2}\frac{|\zeta|}{1-r^2|\zeta|^2}\,dr=\frac{1}{2}\operatorname{arctanh}|\zeta|.
\eeqq
On the other hand, since \(0\leq s(\xi)<1\) for \(\xi\in\mathbb D\), we also have
\beqq
|s(\zeta)-s(0)|\leq1.
\eeqq
Therefore,
\be\label{ggp-0}
|s(\zeta)-s(0)|
\leq
\min\left\{\frac{1}{2}\operatorname{arctanh}|\zeta|,1\right\}.
\ee

We next show that
\be\label{ggp-1}
\min\left\{\frac{1}{2}\operatorname{arctanh}r,1\right\}
\leq
(\operatorname{coth}2)r,
\qquad
0\leq r<1.
\ee
The case \(r=0\) is immediate. Suppose first that \(0<r\leq \tanh2\). Then
\beqq
\frac{1}{2}\operatorname{arctanh}r\leq1.
\eeqq
Moreover, the function
\beqq
F(r)=\frac{\operatorname{arctanh}r}{r}
\eeqq
is increasing on \((0,1)\). To see this, note that
\beqq
F'(r)
=
\frac{\frac{r}{1-r^2}-\operatorname{arctanh}r}{r^2}.
\eeqq
Since
\beqq
\operatorname{arctanh}r
=
\int_0^r\frac{dt}{1-t^2}
\leq
\int_0^r\frac{dt}{1-r^2}
=
\frac{r}{1-r^2},
\eeqq
we have \(F'(r)\geq0\). Consequently, for \(0<r\leq\tanh2\),
\beqq
\frac{\operatorname{arctanh}r}{r}
\leq
\frac{\operatorname{arctanh}(\tanh2)}{\tanh2}
=
\frac2{\tanh2},
\eeqq
which yields that
\beqq
\frac{1}{2}\operatorname{arctanh}r
\leq
\frac{r}{\tanh2}
=
(\operatorname{coth}2)r.
\eeqq
If instead \(\tanh2\leq r<1\), then
\beqq
\min\left\{\frac{1}{2}\operatorname{arctanh}r,1\right\}=1
\leq
\frac{r}{\tanh2}
=
(\operatorname{coth}2)r.
\eeqq
Thus (\ref{ggp-1}) holds.

Combining (\ref{ggp-0}) and (\ref{ggp-1}) with the choice $r=|\zeta|$,  we obtain
\beqq
|s(\zeta)-s(0)|
\leq
(\operatorname{coth}2)|\zeta|,
\eeqq
which, together  with \eqref{3.t3} and \eqref{3.3}, gives that
\be\label{yh0-1}
|E(\zeta)-E(0)|
\leq
\frac{3\sqrt{3}}{2}|\zeta|\,\|G\|_{\mathscr{B}_{s}}
+
b(0)(\operatorname{coth}2)|\zeta|.
\ee
Note that $b(0)\leq\|G\|_{\mathscr{B}_{s}}$. Then, by (\ref{yh0-1}), we have
\beqq
|E(\zeta)-E(0)|
\leq
\left(
\frac{3\sqrt{3}}{2}
+
\operatorname{coth}2
\right)
|\zeta|\,\|G\|_{\mathscr{B}_{s}}.
\eeqq
It follows from the M\"obius invariance of the
analytic Bloch semi-norm again and Lemma \ref{lem-2.2} that
\beqq
\|G\|_{\mathscr{B}_{s}}
=
\|g\circ\varphi_w\|_{\mathscr{B}_{s}}
=
\|g\|_{\mathscr{B}_{s}}
\leq
\|f\|_{\mathscr{B}_{h^*,s}}  \frac{\sqrt{K - 1}}{2}
+ \frac{\sqrt{K_0}}{2}.
\eeqq
Consequently,
\beq\label{3.5}
|E(\zeta)-E(0)|
\leq
\left(
\frac{3\sqrt{3}}{2}
+
\operatorname{coth}2
\right)
\left(
\|f\|_{\mathscr{B}_{h^*,s}}  \frac{\sqrt{K - 1}}{2}
+ \frac{\sqrt{K_0}}{2}
\right)|\zeta|.
\eeq

From \eqref{3.1}, \eqref{3.2} and \eqref{3.5},  we conclude that
\begin{align*}
|U(\zeta)-U(0)|
&\leq
\left[
C_{0}\frac{\sqrt{ K+3}}{2}
+
\left(
\frac{3\sqrt{3}}{2}
+
\operatorname{coth}2
\right)
  \frac{\sqrt{K - 1}}{2}
\right]
\|f\|_{\mathscr{B}_{h^*,s}}|\zeta| \\
&\quad+
\left[
C_{0}
+
\left(
\frac{3\sqrt{3}}{2}
+
\operatorname{coth}2\right)\right]
\frac{\sqrt{K_{0}}}{2}|\zeta|.
\end{align*}
The proof of this theorem is complete.
\qed

\blem\label{lem-2.5}
Let \(f=h+\overline g\) be a harmonic \(K\)-quasiregular mapping in
\(\mathbb D\), where \(K\geq1\). If
\(f\in \mathscr{B}_h^*\), then
\beqq
\|g\|_{\mathscr{B}_{s}}
\leq
\|f\|_{\mathscr{B}_{h^*,s}} \frac{K-1}{2\sqrt{K}}
\eeqq
and
\beqq
\|h\|_{\mathscr{B}_{s}}
\leq
\|f\|_{\mathscr{B}_{h^*,s}} \frac{K+1}{2\sqrt{K}}.
\eeqq
\elem

\begin{proof}
Let  \(M \), $a(z)$ and $b(z)$ be as in the proof of Lemma  \ref{lem-2.2}.
By the similar reasoning to that in the proof of Lemma  \ref{lem-2.2}, we get
\begin{align*}
(a(z) + b(z))^2
&\leq K\bigl(a(z)^2 - b(z)^2\bigr).
\end{align*}
Equivalently,
\beqq
(K + 1) b(z)
\leq (K - 1) a(z).
\eeqq
Combining this with \eqref{eq-2.19}, we have
\beqq
(K + 1)^2 b(z)^2
\leq (K - 1)^2\bigl(M^2 + b(z)^2\bigr) .
\eeqq
Thus,
\beq\label{cq-chj2}
4K b(z)^2 \leq (K - 1)^2 M^2 ,
\eeq
and consequently,
\beqq
b(z)
\leq M \frac{K-1}{2\sqrt{K}}.
\eeqq
Taking the supremum over \(z \in \mathbb D\), we obtain
\beqq
\|g\|_{\mathscr{B}_{s}}=\sup_{z \in \mathbb D} b(z)
= \sup_{z \in \mathbb D} \mathcal{B}_{g}(z)
\leq M \frac{K-1}{2\sqrt{K}}.
\eeqq
On the other hand, it follows from \eqref{eq-2.19} and \eqref{cq-chj2} that
$$
4a(z)^2 \leq 4 M^2 + 4b(z)^2  \leq 4M^2+\frac{(K - 1)^2}{K} M^2  = \frac{(K+1)^2}{K}M^2,
$$
which implies that
\beqq
a(z)
\leq M \frac{K+1}{2\sqrt{K}}.
\eeqq
Taking the supremum over \(z \in \mathbb D\), we  obtain
\beqq
\|h\|_{\mathscr{B}_{s}}
= \sup_{z \in \mathbb D} a(z)
= \sup_{z \in \mathbb D} \mathcal{B}_{h}(z)
\leq  M \frac{K+1}{2\sqrt{K}}.
\eeqq
The proof of this lemma is complete.
\end{proof}

\subsection*{Proof of Theorem \ref{cor-1}}
We will use an argument similar to that in the proof of Theorem~\ref{thm-1}.
Let $U(\zeta)$, $a(\zeta)$, $b(\zeta)$, $s(\zeta)$, $\omega(\zeta)$ be as in the proof of Theorem~\ref{thm-1}.
It suffices to estimate
\beqq
|U(\zeta)-U(0)|.
\eeqq
By the proof of Theorem ~\ref{thm-1},
\beq\label{Eq-2.30}
|U(\zeta)-U(0)|
\leq
|a(\zeta)-a(0)|+|b(\zeta)-b(0)|s(\zeta)
+
b(0)|s(\zeta)-s(0)|.
\eeq
Since $f$ is  locally univalent harmonic $K$-quasiregular,
we have
\begin{align}\label{Eq-2.31}
|\omega(\zeta)|\leq \frac{K-1}{K+1}=k.
\end{align}
Therefore, 
we obtain 
\begin{align}\label{Eq-2.32}
s(\zeta)=\frac{|\omega(\zeta)|}{1+\sqrt{1-|\omega(\zeta)|^2}}\leq \frac{k}{1+\sqrt{1-k^2}}=\frac{kL(K)}{2}.
\end{align}

Next, 
by Schwarz--Pick's lemma for $\omega/k$,
\beqq
|\omega'(r\zeta)|
\leq
\frac{k(1-|\omega(r\zeta)|^2)}
{1-r^2|\zeta|^2}.
\eeqq
Then as in the proof of Theorem~\ref{thm-1},
we obtain
\beqq
|s(\zeta)-s(0)|
\leq
\frac{k}{2}\operatorname{arctanh}|\zeta|.
\eeqq
On the other hand, since \(0\leq s(\xi)\leq \frac{kL(K)}{2}\) for \(\xi\in\mathbb D\), we also have
\beqq
|s(\zeta)-s(0)|\leq\frac{kL(K)}{2}.
\eeqq
Consequently,
\beqq
|s(\zeta)-s(0)|
\leq
k\min\left\{\frac{1}{2}\operatorname{arctanh}|\zeta|,\frac{L(K)}{2}\right\}.
\eeqq

As in the proof of  Theorem~\ref{thm-1}, we obtain the elementary estimate
\beqq
\min\left\{\frac{1}{2}\operatorname{arctanh}r,\frac{L(K)}{2}\right\}
\leq
\frac{L(K)}{2}(\operatorname{coth}L(K))r,
\qquad
0\leq r<1.
\eeqq

Taking \(r=|\zeta|\), we get
\begin{align}\label{Eq-2.33}
|s(\zeta)-s(0)|
&\leq
\frac{kL(K)}{2}(\operatorname{coth}L(K))|\zeta|.
\end{align}
Note that \(b(0)\leq
\|g\|_{\mathscr{B}_{s}}\),
where ${g}$ is the complex conjugate of the coanalytic part of $f$. 
Substituting
 \eqref{chj0}, \eqref{3.3}, \eqref{Eq-2.32} and \eqref{Eq-2.33} into \eqref{Eq-2.30},
we obtain that
\begin{align*}
|U(\zeta)-U(0)|
&\leq
C_{0}|\zeta|
\|h\|_{\mathscr{B}_{s}}
+
\frac{kL(K)}{2}\left(
\frac{3\sqrt{3}}{2}
+
\operatorname{coth}L(K)
\right)
|\zeta|
\|g\|_{\mathscr{B}_{s}}.
\end{align*}
Thus, by Lemma \ref{lem-2.5}, we obtain
\begin{align*}
|U(\zeta)-U(0)|
&\leq
\left[
C_{0}\frac{(K+1)}{2\sqrt{K}}
+
\frac{kL(K)}{2}\left(
\frac{3\sqrt{3}}{2}
+
\operatorname{coth}L(K)
\right)
  \frac{(K - 1)}{2\sqrt{K}}
\right]
\|f\|_{\mathscr{B}_{h^*,s}}|\zeta|.
\end{align*}
The proof of this theorem is complete.
\qed

\bigskip
{\bf Data Availability} Our manuscript has no associated data.

{\bf Conflict of interest} The authors declare that they have no conflict of interest.

\bigskip

{\rm $\mathbf{Acknowledgments:}$}	The first author was partially supported by the National Natural Science Foundation of China (Grant No. 12571080) and the Guangxi Natural Science Foundation (Grant No. 2026GXNSFFA00640002). The second author was partially supported by JSPS KAKENHI Grant Number JP22K03363. The third author was partially supported by the National Natural Science Foundation of China (Grant No. 12371071) and the Key Project of the NSF of Hunan Province (Grant No. 2026JJ30002). The fourth author was partially supported by the Scientific Research Fund of the Hunan Provincial Education Department (Grant No. 25A0086).

\end{document}